\documentclass[a4paper,12pt,reqno]{amsart}

\usepackage{amsmath,amsthm,amssymb}
\usepackage[top=30truemm,bottom=30truemm,left=25truemm,right=25truemm]{geometry}

\usepackage{color}

\newtheorem{thm}{Theorem}[subsection]
\newtheorem{lem}[thm]{Lemma}
\newtheorem{prop}[thm]{Proposition}
\newtheorem{cor}[thm]{Corollary}
\newtheorem{conj}[thm]{Conjecture}

\theoremstyle{definition}
\newtheorem{defi}[thm]{Definition}
\newtheorem{rem}[thm]{Remark}
\newtheorem{ex}[thm]{Example}

\numberwithin{equation}{section}

\begin{document}
\author{Hideya Watanabe}
\address{(H. Watanabe) Osaka Central Advanced Mathematical Institute, Osaka Metropolitan University, Osaka, 558-8585, Japan}
\email{watanabehideya@gmail.com}
\date{\today}

\title{Stability of $\imath$canonical bases of locally finite type}

\maketitle

\begin{abstract}
  We prove the stability conjecture of $\imath$canonical bases, which was raised by Huanchen Bao and Weiqiang Wang in 2016, for all locally finite types.
  To this end, we characterize the trivial module over the $\imath$quantum groups of such type at $q = \infty$.
  This result can be seen as a very restrictive version of the $\imath$crystal base theory for locally finite types.
\end{abstract}

\section{Introduction}
\subsection{Quantum symmetric pairs}
A quantum symmetric pair is a pair $(\mathbf{U}, \mathbf{U}^\imath)$ of a quantum group (also known as a quantized enveloping algebra) $\mathbf{U}$, introduced by Drinfeld \cite{Dr85} and Jimbo \cite{Jim85}, and an $\imath$quantum group (also known as a quantum symmetric pair coideal subalgebra) $\mathbf{U}^\imath$, uniformly formulated by Letzter \cite{Let99} for finite types and by Kolb \cite{Kol14} for Kac-Moody types.
The $\imath$quantum group $\mathbf{U}^\imath$ is constructed from an admissible pair, which is a generalization of a Satake diagram.
When the admissible pair is of diagonal type, the $\imath$quantum group $\mathbf{U}^\imath$ is an ordinary quantum group.
In this sense, the $\imath$quantum groups are generalizations of the quantum groups.

Based on this point of view, many important constructions for the quantum groups have been generalized to the $\imath$quantum groups in natural, but often nontrivial, ways ({\it cf.}\ \cite{Wan21}).
One of them is Bao-Wang's construction of the $\imath$canonical basis of the modified form of $\mathbf{U}^\imath$, which generalizes Lusztig's construction of the canonical basis of the modified form of $\mathbf{U}$.

\subsection{Lusztig's construction}
The canonical basis $\dot{\mathbf{B}}$ of the modified form $\dot{\mathbf{U}}$ of the quantum group $\mathbf{U}$ was constructed by Lusztig \cite{Lus92}.
It is a projective limit of a certain projective system of $\mathbf{U}$-modules.
To be more specific, let $X^+$ denote the set of dominant weights for the quantum group $\mathbf{U}$.
For each $\lambda \in X^+$, let $V(\lambda)$ and $V(-\lambda)$ denote the integrable highest weight module of highest weight $\lambda$ with highest weight vector $v_\lambda$ and the integrable lowest weight module of lowest weight $-\lambda$ with lowest weight vector $v_{-\lambda}$, respectively.
In \cite{Lus92}, Lusztig defined the canonical basis $\mathbf{B}(-\lambda, \mu)$ of the tensor product $V(-\lambda, \mu) := V(-\lambda) \otimes V(\mu)$ for each $\lambda,\mu \in X^+$.
Then, he showed that these canonical bases are stable under the $\mathbf{U}$-module homomorphisms
\[
  \pi_{\lambda,\mu;\nu}: V(-(\lambda+\nu), \mu+\nu) \rightarrow V(-\lambda, \mu)
\]
sending $v_{-(\lambda + \nu)} \otimes v_{\mu+\nu}$ to $v_{-\lambda} \otimes v_\mu$ for each $\lambda,\mu,\nu \in X^+$.
Namely, for each $b \in \mathbf{B}(-(\lambda+\nu), \mu+\nu)$, the image $\pi_{\lambda,\mu;\nu}(b)$ is either an element of $\mathbf{B}(-\lambda, \mu)$ or zero, and the kernel of $\pi_{\lambda,\mu;\nu}$ is spanned by a subset of $\mathbf{B}(-(\lambda+\nu), \mu+\nu)$.
In other words, the homomorphisms $\pi_{\lambda,\mu;\nu}$ are based.

Also, he proved that $\dot{\mathbf{B}}$ is stable under the $\mathbf{U}$-module homomorphisms
\[
  \dot{\pi}_{\lambda,\mu;\nu}: \dot{\mathbf{U}} \rightarrow V(-(\lambda+\nu), \mu+\nu);\ x \mapsto x \cdot (v_{-(\lambda+\nu)} \otimes v_{\mu+\nu})
\]
for all $\lambda,\mu,\nu \in X^+$.

\subsection{Bao-Wang's construction}
Bao and Wang generalized Lusztig's construction to $\imath$quantum groups for finite types in \cite{BW18} and for Kac-Moody types in \cite{BW21}.
Namely, they defined $\mathbf{U}^\imath$-modules $V^\imath(\lambda,\mu;\nu)$ and their $\imath$canonical bases $\mathbf{B}^\imath(\lambda,\mu;\nu)$ for $\lambda,\mu,\nu \in X^+$, which are counterparts of $V(-(\lambda+\nu), \mu+\nu)$ with canonical bases $\mathbf{B}(-(\lambda+\nu),\mu+\nu)$.
Also, they defined $\mathbf{U}^\imath$-module homomorphisms
\[
  \pi^\imath_{\lambda,\mu;\nu}: V^\imath(\lambda,\mu;\nu) \rightarrow V^\imath(\lambda,\mu;0)
\]
and
\[
  \dot{\pi}^\imath_{\lambda,\mu;\nu}: \dot{\mathbf{U}}^\imath \rightarrow V^\imath(\lambda,\mu;\nu)
\]
for each $\lambda,\mu,\nu \in X^+$, where $\dot{\mathbf{U}}^\imath$ denotes the modified form of $\mathbf{U}^\imath$.
Then, they showed that the $\imath$canonical bases $\mathbf{B}^\imath(\lambda,\mu;\nu)$ are \emph{asymptotically} stable under the homomorphisms $\pi^\imath_{\lambda,\mu;\nu}$, and constructed the $\imath$canonical basis $\dot{\mathbf{B}}^\imath$ of $\dot{\mathbf{U}}^\imath$, which, together with $\dot{\pi}^\imath_{\lambda,\mu;\nu}$, is an \emph{asymptotical} limit of the projective system $V^\imath(\lambda,\mu;\nu)$ with $\pi^\imath_{\lambda,\mu;\nu}$.

\subsection{Stability conjecture}
Bao-Wang's construction of the $\imath$canonical basis of $\dot{\mathbf{U}}^\imath$ is quite natural, and the $\imath$canonical bases of $\dot{\mathbf{U}}^\imath$ and $V^\imath(\lambda,\mu;\nu)$ have many good properties as the canonical bases of $\dot{\mathbf{U}}$ and $V(-(\lambda+\nu), \mu+\nu)$.
However, $\imath$canonical bases are not necessarily stable under the homomorphisms $\pi^\imath_{\lambda,\mu;\nu}$ and $\dot{\pi}^\imath_{\lambda,\mu;\nu}$ in general.
Bao and Wang conjectured that if the parameters for $\mathbf{U}^\imath$ are chosen appropriately, then the $\imath$canonical bases are stable \cite[Remark 6.18]{BW18}.
We call it the \emph{stability conjecture}.

The stability of $\imath$canonical bases are closely related to the study of categorifications of $\imath$quantum groups, quantum symmetric pairs at roots of unity, and symmetric subgroup schemes (see \cite{BSWW18}, \cite{BSa21}, \cite{BS22}, and \cite{BWW23a}).

The stability conjecture has been proved to be true when the admissible pair is of locally finite type and quasi-split in \cite{W21}, and when the admissible pair is irreducible, of finite type, and of real rank $1$ in \cite{W23}.
In both of the two papers, the final step in the proof of the stability conjecture is to show that certain $\mathbf{U}$-modules contain a $\mathbf{U}^\imath$-submodule isomorphic to the trivial module generated by a vector which coincides with a highest weight vector at $q = \infty$.
In \cite{W21}, this claim was proved by means of the $\imath$crystal base theory, which is a $\mathbf{U}^\imath$-analogue of the crystal base theory for $\mathbf{U}$, while in \cite{W23}, by direct calculation based on the classification of the admissible pairs under consideration.

The $\imath$crystal base theory has not been established beyond quasi-split locally finite types.
Hence, the strategy of \cite{W21} is not applicable for a general type.
On the other hand, the strategy of \cite{W23} is valid for all types.
However, we cannot complete the proof in this way because there are infinitely many types of admissible pairs.

\subsection{Results}
In this paper, we prove the stability conjecture when the admissible pair is of locally finite type.
Such admissible pairs include all and much more ones treated in \cite{W21} or \cite{W23}.

As stated above, in order to prove the stability conjecture, we only need to understand the trivial module at $q = \infty$.
To this end, we characterize the trivial $\mathbf{U}^\imath$-submodules in each integrable highest weight $\mathbf{U}$-module $V(\lambda)$ in terms of its (ordinary) crystal base.
This can be seen as a very restrictive version of the $\imath$crystal base theory for locally finite types by the following reason.

One of the goals of the $\imath$crystal base theory is to understand various $\mathbf{U}^\imath$-modules from their behaviours at $q = \infty$, just like the ordinary crystal base theory characterizes each irreducible component in the $\mathbf{U}$-modules of the form $V(-\lambda, \mu)$ at $q = \infty$ in terms of their crystal bases, which can be deduced from the crystal bases of $V(-\lambda, \mu)$ as $\mathbf{U} \otimes \mathbf{U}$-modules by the tensor product rule.
This goal has been achieved for all diagonal types and for types $A\mathrm{I}_1$, $A\mathrm{III}_2$, and $A\mathrm{IV}_2$ in \cite{W21}, and for type $A\mathrm{I}$ of arbitrary rank in \cite{W23a} and \cite{W21b}.
The characterization of the trivial submodules mentioned above partially achieve the goal for all locally finite types.

\subsection{Organization}
This paper is organized as follows.
In Sections \ref{sect: qg and cb} and \ref{sect: qsp and cb}, we recall basic results and fix our notations concerning quantum groups and quantum symmetric pairs, respectively.
The stability conjecture is precisely stated in Subsection \ref{subsect: stab conj}.
Then, we prove the stability conjecture for locally finite types in Section \ref{sect: proof of stab conj}.

\subsection{Acknowledgements}
This work was supported by JSPS KAKENHI Grant Number JP20K14286.

\section{Quantum groups and canonical bases}\label{sect: qg and cb}
In this section, we recall basic results and fix our notations concerning quantum groups and canonical bases of various modules.

\subsection{Quantum Groups}
Let $I$ be a Cartan datum, and $(Y, X, \langle , \rangle, \Pi^\vee, \Pi)$ a $Y$-regular and $X$-regular root datum of type $I$ in the sense of \cite[Sections 1.1 and 2.2]{Lus93}.
Namely, $A = (a_{i,j})_{i,j \in I}$ is a symmetrizable generalized Cartan matrix, $Y$ and $X$ are free abelian groups with perfect bilinear pairing
\[
  \langle ,  \rangle: Y \times X \rightarrow \mathbb{Z},
\]
$\Pi^\vee = \{ h_i \mid i \in I \} \subset Y$ and $\Pi = \{ \alpha_i \mid i \in I \} \subset X$ are linearly independent subsets such that
\[
  \langle h_i, \alpha_j \rangle = a_{i,j}\ \text{ for all } i,j \in I.
\]

Let us fix an $I$-tuple $(d_i)_{i \in I}$ of positive integers in a way such that
\[
  d_i a_{i,j} = d_j a_{j,i}\ \text{ for all } i,j \in I.
\]
For each $i \in I$ and $n \in \mathbb{Z}_{\geq 0}$, set
\[
  q_i := q^{d_i}, \quad [n]_i := \frac{q_i^n-q_i^{-n}}{q_i-q_i^{-1}}, \quad [n]_i! := \prod_{k=1}^n [k]_i.
\]

The quantum group associated with the datum above is the unital associative $\mathbb{Q}(q)$-algebra $\mathbf{U}$ with generators
\[
  \{ E_i, F_i, K_h \mid i \in I,\ h \in Y \}
\]
subject to the following relations for all $i,j \in I$ and $h,h_1,h_2 \in Y$:
\begin{align*}
  \begin{split}
    &K_0 = 1, \\
    &K_{h_1} K_{h_2} = K_{h_1 + h_2}, \\
    &K_h E_i = q^{\langle h, \alpha_i \rangle} E_i K_h, \\
    &K_h F_i = q^{\langle h, -\alpha_i \rangle} F_i K_h, \\
    &E_i F_j - F_j E_i = \delta_{i,j} \frac{K_i - K_i^{-1}}{q_i - q_i^{-1}}, \\
    &\sum_{r+s=1-a_{i,j}} (-1)^s E_i^{(r)} E_j E_i^{(s)} = \sum_{r+s=1-a_{i,j}} (-1)^s F_i^{(r)} F_j F_i^{(s)} = 0 \quad \text{ if } i \neq j,
  \end{split}
\end{align*}
where
\[
  K_i := K_{d_i h_i},\ E_i^{(n)} := \frac{1}{[n]_i!} E_i^n,\ F_i^{(n)} := \frac{1}{[n]_i!} F_i^n\ \text{ for each } i \in I \text{ and } n \in \mathbb{Z}_{\geq 0}.
\]

The quantum group $\mathbf{U}$ has the following Hopf algebra structure with multiplication $\Delta$, counit $\epsilon$, and antipode $S$ (\cite[3.3.4]{Lus93}):
\begin{align*}
  \begin{split}
    &\Delta(E_i) = E_i \otimes 1 + K_i \otimes E_i, \\
    &\Delta(F_i) = 1 \otimes F_i + F_i \otimes K_i^{-1}, \\
    &\Delta(K_h) = K_h \otimes K_h, \\
    &\epsilon(E_i) = \epsilon(F_i) = 0, \quad \epsilon(K_h) = 1, \\
    &S(E_i) = -K_i^{-1} E_i, \quad S(F_i) = -F_i K_i^{-1}, \quad S(K_h) = K_{-h}.
  \end{split}
\end{align*}

The \emph{trivial module} is a one-dimensional $\mathbf{U}$-module defined by the counit $\epsilon: \mathbf{U} \rightarrow \mathbb{Q}(q)$.

Let $\overline{\cdot}$ denote the bar-involution on $\mathbf{U}$.

By \cite[19.1.1]{Lus93}, there exists a unique anti-algebra involution $\rho$ on $\mathbf{U}$ such that
\[
  \rho(E_i) = q_i K_i F_i, \quad \rho(F_i) = q_i K_i^{-1} E_i, \quad \rho(K_h) = K_h.
\]

Let $W$ denote the Weyl group of $I$ with simple reflections $s_i$ for $i \in I$.

For each $i \in I$, let $T_i$ denote the algebra automorphism $T''_{i,1}$ on $\mathbf{U}$ in \cite[Section 37.1]{Lus93}.
The automorphisms $(T_i)_{i \in I}$ satisfy the braid relations of type $I$.
Hence, for each $w \in W$ with a reduced expression $w = s_{i_1} \cdots s_{i_r}$, the composition
\[
  T_w := T_{i_1} \cdots T_{i_r}
\]
is well-defined.

\subsection{Integrable highest weight modules}
Let us recall some terminologies concerning $\mathbf{U}$-modules.

\begin{defi}\normalfont\hfill
  \begin{enumerate}
    \item A $\mathbf{U}$-module $M$ is said to be a \emph{weight module} if it possesses a weight space decomposition:
    \[
      M = \bigoplus_{\lambda \in X} M_\lambda, \quad M_\lambda := \{ m \in M \mid K_h m = q^{\langle h, \lambda \rangle} m \ \text{ for all } h \in Y \}.
    \]
    The vectors in $M_\lambda$ are called the \emph{weight vectors of weight $\lambda$}.
    \item A $\mathbf{U}$-module $M$ is said to be a \emph{highest weight module} if there exists a nonzero weight vector $m \in M$ of some weight $\lambda \in X$ satisfying the following:
    \begin{itemize}
      \item $E_i m = 0$ for all $i \in I$.
      \item $M = \mathbf{U} m$.
    \end{itemize}
    The weight $\lambda$ and the vector $m$ are referred to as the \emph{highest weight} and a \emph{highest weight vector}, respectively.
    \item A weight $\mathbf{U}$-module $M$ is said to be an \emph{integrable module} if for each $m \in M$, there exists a positive integer $N$ such that
    \[
      E_{i_1} \cdots E_{i_N} m = F_{i_1} \cdots F_{i_N} m = 0 \ \text{ for all } i_1,\dots,i_N \in I.
    \]
    \item A bilinear form $(,)$ on a $\mathbf{U}$-module $M$ is said to be \emph{contragredient} if
    \[
      (xu, v) = (u, \rho(x) v) \ \text{ for all } x \in \mathbf{U} \text{ and } u,v \in M.
    \]
  \end{enumerate}
\end{defi}

For each $\lambda \in X$, let $M(\lambda)$ denote the Verma module of highest weight $\lambda$ with highest weight vector $m_\lambda$.
It is isomorphic to the quotient module
\[
  \mathbf{U}/\left( \sum_{h \in Y} \mathbf{U}(K_h - q^{\langle h, \lambda \rangle}) + \sum_{i \in I} \mathbf{U} E_i \right).
\]

Let $X^+$ denote the set of dominant weights:
\[
  X^+ := \{ \lambda \in X \mid \langle h_i, \lambda \rangle \geq 0 \ \text{ for all } i \in I \}.
\]
For each $\lambda \in X^+$, there exists a unique integrable highest weight module $V(\lambda)$ of highest weight $\lambda$ with highest weight vector $v_\lambda$.
It is isomorphic to the quotient module
\[
  M(\lambda)/\sum_{i \in I} \mathbf{U} F_i^{(\langle h_i, \lambda \rangle + 1)} m_\lambda.
\]
Let $\mathbf{B}(\lambda)$ and $(\mathcal{L}(\lambda), \mathcal{B}(\lambda))$ denote the canonical basis and crystal base (at $q = \infty$) of $V(\lambda)$, respectively.

The trivial module is isomorphic to $V(0)$ under the isomorphism mapping $1$ to $v_0$.
We shall identify these two modules in this way.

By \cite[Proposition 19.1.2]{Lus93}, there exists a unique contragredient symmetric bilinear form $(,)$ on $V(\lambda)$ such that
\[
  (v_\lambda, v_\lambda) = 1.
\]
The canonical basis $\mathbf{B}(\lambda)$ forms an almost orthonormal basis with respect to this form (\cite[Proposition 19.3.3]{Lus93}):
\[
  (b, b') \in \delta_{b,b'} + q^{-1} \mathbb{Z}[q^{-1}]\ \text{ for all } b,b' \in \mathbf{B}(\lambda).
\]

\begin{thm}[{See e.g.\ \cite[Proposition 6.3.6]{Lus93}}]\label{thm: comp red for fin}
Assume that the Cartan datum $I$ is of finite type.
Then, each integrable $\mathbf{U}$-module is completely reducible with irreducible components isomorphic to $V(\lambda)$ with various $\lambda \in X^+$.
\end{thm}

\subsection{Based modules}
Let $\mathcal{A} := \mathbb{Z}[q,q^{-1}]$ denote the ring of Laurent polynomials with coefficients in $\mathbb{Z}$.
Let $\dot{\mathbf{U}}$ denote the modified form of $\mathbf{U}$, and ${}_\mathcal{A} \dot{\mathbf{U}}$ the integral form.
As shown in \cite[23.1.4]{Lus93}, each weight $\mathbf{U}$-module has a natural $\dot{\mathbf{U}}$-module structure.

Let $\mathbf{A}_\infty$ denote the subring of $\mathbb{Q}(q)$ consisting of all rational functions regular at $q = \infty$.

\begin{defi}\normalfont
  A weight $\mathbf{U}$-module $M$ equipped with a basis $B$ is said to be \emph{based} if it satisfies the following:
  \begin{enumerate}
    \item The set $B \cap M_\lambda$ is a basis of $M_\lambda$ for all $\lambda \in X$.
    \item The $\mathcal{A}$-submodule ${}_\mathcal{A} M$ of $M$ spanned by $B$ is stable under the action of ${}_\mathcal{A} \dot{\mathbf{U}}$.
    \item The $\mathbb{Q}$-linear endomorphism $\overline{\cdot}$ on $M$, defined by $\overline{q^n b} = q^{-n} b$ for all $n \in \mathbb{Z}$ and $b \in B$, is compatible with the bar-involution on $\mathbf{U}$:
    \[
      \overline{xm} = \bar{x} \bar{m} \quad \text{ for all } x \in \mathbf{U},\ m \in M.
    \]
    \item Set $\mathcal{L} := \mathbf{A}_\infty B$ and $\mathcal{B} := \{ b + q^{-1} \mathcal{L} \mid b \in B \}$.
    Then, the set $\mathcal{B}$ is a $\mathbb{Q}$-basis of $\mathcal{L}/q^{-1} \mathcal{L}$.
  \end{enumerate}
\end{defi}

\begin{defi}\normalfont
  A $\mathbf{U}$-module homomorphism $f : M \rightarrow N$ between two based modules $(M, B_M)$ and $(N, B_N)$ is said to be \emph{based} if $f(B_M) \subseteq B_N \sqcup \{0\}$ and the kernel of $f$ is spanned by a subset of $B_M$.
\end{defi}

\begin{defi}\normalfont
  A $\mathbf{U}$-submodule $N$ of a based module $M$ is said to be a \emph{based submodule} if the inclusion map $N \hookrightarrow M$ is a based homomorphism.
\end{defi}

Given a based module $(M, B)$ and two vectors $u,v \in M$, we write $u \equiv_\infty v$ to indicate that
\[
  u-v \in q^{-1} \mathcal{L}.
\]

\begin{ex}\label{ex: characterization of the hwv at infty}\normalfont
  Let $\lambda \in X^+$.
  The integrable highest weight module $V(\lambda)$ with its canonical basis $\mathbf{B}(\lambda)$ forms a based module.
  For each $v \in \mathcal{L}(\lambda)$, we have
  \[
    \tilde{E}_i v \equiv_\infty 0 \text{ for all } i \in I \text{ if and only if } v \equiv_\infty c v_\lambda \text{ for some } c \in \mathbb{Q},
  \]
  where $\tilde{E}_i$ denotes the Kashiwara operator.
\end{ex}

\begin{ex}\normalfont
  Let $\lambda, \mu \in X^+$.
  Set
  \[
    \mathbf{B}(\lambda) \diamondsuit \mathbf{B}(\mu) := \{ b_1 \diamondsuit b_2 \mid (b_1, b_2) \in \mathbf{B}(\lambda) \times \mathbf{B}(\mu) \},
  \]
  where $b_1 \diamondsuit b_2$ is as in \cite[Theorem 2.7 (1)]{BW16}.
  Then, the tensor product $V(\lambda) \otimes V(\mu)$ with $\mathbf{B}(\lambda) \diamondsuit \mathbf{B}(\mu)$ is a based module (\cite[Theorem 2.7]{BW16}).
\end{ex}

\begin{ex}\normalfont
  For each $\lambda, \mu \in X^+$ and $w \in W$, let $V(w\lambda, \mu)$ denote the $\mathbf{U}$-submodule of $V(\lambda) \otimes V(\mu)$ generated by $v_{w\lambda} \otimes v_\mu$, where $v_{w\lambda}$ denote the unique element in $\mathbf{B}(\lambda)$ of weight $w\lambda$.
  Then, the $\mathbf{U}$-submodule $V(w\lambda, \mu)$ is a based submodule of $(V(\lambda) \otimes V(\mu), \mathbf{B}(\lambda) \diamondsuit \mathbf{B}(\mu))$ (\cite[Theorem 2.2]{BW21}).
  Set
  \[
    \mathbf{B}(w\lambda, \mu) := V(w\lambda, \mu) \cap \mathbf{B}(\lambda) \diamondsuit \mathbf{B}(\mu).
  \]
\end{ex}

\section{Quantum symmetric pairs and $\imath$canonical bases}\label{sect: qsp and cb}
In this section, we recall basic results and fix our notations concerning quantum symmetric pairs and $\imath$canonical bases of various objects.
Then, we state the stability conjecture.

\subsection{Admissible pairs and Satake data}
Let $(I_\bullet, \tau)$ be an admissible pair of type $I$ in the sense of \cite[Definition 2.3]{Kol14}.
In particular, the $I_\bullet$ is a subdatum of the Cartan datum $I$ of finite type, and the $\tau$ is an automorphism on $I$ of order at most $2$.

\begin{rem}\normalfont
  Regelskis and Vlaar \cite{RV20} generalized the notion of admissible pairs to that of generalized Satake diagrams by weakening one of the axioms for admissible pairs.
  Most of this paper is valid for all generalized Satake diagrams because we do not use the axiom.
\end{rem}

We use the following notation throughout the paper.
\begin{itemize}
  \item $W_\bullet \subseteq W$: the subgroup generated by $s_j$ for all $j \in I_\bullet$.
  It is the Weyl group of the Cartan datum $I_\bullet$.
  \item $w_\bullet \in W_\bullet$: the longest element of $W_\bullet$.
  \item $\rho^\vee_\bullet \in Y$: half the sum of all positive coroots of $I_\bullet$.
  \item $\rho_\bullet \in X$: half the sum of all positive roots of $I_\bullet$.
  \item $I_\circ := I \setminus I_\bullet$.
  \item $I_\circ^{\tau, \bullet} := \{ i \in I_\circ \mid \tau(i) = i \text{ and } \langle h_j, \alpha_i \rangle = 0 \text{ for all } j \in I_\bullet \}$, or equivalently,
  \[
    I_\circ^{\tau, \bullet} := \{ i \in I \mid h_i - w_\bullet h_{\tau(i)} = 0 \}.
  \]
\end{itemize}

\begin{defi}\label{def: related to adm pair}\normalfont\hfill
  \begin{enumerate}
    \item The \emph{real rank} of $(I_\bullet, \tau)$ is the number of $\tau$-orbits in $I_\circ$.
    \item Given $i,j \in I$, we say that $j$ is \emph{connected} to $i$ if there exists a sequence $i = i_1, i_2, \dots, i_r = j$ such that $a_{i_k, i_{k+1}} \neq 0$ for all $k = 1,\dots,r-1$.
    \item $(I_\bullet, \tau)$ is said to be \emph{irreducible} if for each $i,j \in I$, the $j$ is connected to either $i$ or $\tau(i)$.
    \item\label{item: real rank 1 comp} The \emph{real rank $1$ component} of $i \in I_\circ$ is the subdatum $I_i \subseteq I$ consisting of $i$, $\tau(i)$, and all $j \in I_\bullet$ connected to either $i$ or $\tau(i)$.
    Note that the pair $(I_{i, \bullet} := I_\bullet \cap I_i, \tau|_{I_i})$ is an irreducible admissible pair of real rank $1$ of type $I_i$.
    \item $(I_\bullet, \tau)$ is said to be of \emph{finite type} if the Cartan datum $I$ is of finite type.
    \item $(I_\bullet, \tau)$ is said to be of \emph{locally finite type} if for each $i \in I_\circ$, the Cartan datum $I_i$ is of finite type.
  \end{enumerate}
\end{defi}

\begin{defi}\label{def: satake datum}\normalfont
  The root datum $(Y, X, \langle , \rangle, \Pi^\vee, \Pi)$ equipped with two group automorphisms $\tau$ on the lattices $Y$ and $X$ of order at most $2$, is said to be a \emph{Satake datum} of type $(I_\bullet, \tau)$ if
  \[
    \tau(h_i) = h_{\tau(i)}, \quad \tau(\alpha_i) = \alpha_{\tau(i)} \ \text{ for all } i \in I,
  \]
  and
  \[
    \langle \tau(h), \tau(\lambda) \rangle = \langle h, \lambda \rangle \ \text{ for all } h \in Y, \text{ and } \lambda \in X.
  \]
\end{defi}

From now on, we assume that the root datum $(Y, X, \langle , \rangle, \Pi^\vee, \Pi)$ is a Satake datum of type $(I_\bullet, \tau)$.
Set
\[
  X^\imath := X/\{ \lambda + w_\bullet \tau(\lambda) \mid \lambda \in X \}, \quad Y^\imath := \{ h \in Y \mid h + w_\bullet \tau(h) = 0 \}.
\]
Then, the perfect pairing $\langle ,  \rangle: Y \times X \rightarrow \mathbb{Z}$ induces a (not necessarily perfect) bilinear pairing
\[
  \langle ,  \rangle: Y^\imath \times X^\imath \rightarrow \mathbb{Z}.
\]
For each $\lambda \in X$, let $\overline{\lambda}$ denote the image of $\lambda$ in $X^\imath$.

\subsection{Quantum symmetric pairs}
Let ${\boldsymbol \varsigma} = (\varsigma_i)_{i \in I_\circ} \in (\mathbb{Q}(q)^\times)^{I_\circ}$ and ${\boldsymbol \kappa} = (\kappa_i)_{i \in I_\circ} \in \mathbb{Q}(q)^{I_\circ}$ be parameters satisfying the following for all $i \in I_\circ$.
\begin{align}
  &\varsigma_i, \kappa_i \in \mathcal{A}. \label{eq: constr for paramas of qsp 1} \\
  &\varsigma_{\tau(i)} = \varsigma_i \text{ if } \langle h_i, w_\bullet \alpha_{\tau(i)} \rangle = 0. \label{eq: constr for paramas of qsp 2} \\
  &\varsigma_i = (-1)^{\langle 2\rho_\bullet^\vee, \alpha_i \rangle} q_i^{-\langle h_i, 2\rho_\bullet + w_\bullet \alpha_{\tau(i)} \rangle} \overline{\varsigma_{\tau(i)}}. \label{eq: constr for paramas of qsp 3} \\
  &\overline{\varsigma_i} = \varsigma_i^{-1}. \label{eq: constr for paramas of qsp 4} \\
  &\kappa_i = 0 \text{ unless } i \in I_\circ^{\tau, \bullet} \text{ and } \langle h_k, \alpha_i \rangle \in 2\mathbb{Z} \text{ for all } k \in I_\circ^{\tau, \bullet}. \label{eq: constr for paramas of qsp 5} \\
  &\overline{\kappa_i} = \kappa_i. \label{eq: constr for paramas of qsp 6}
\end{align}

\begin{rem}\label{rem: params for Ui}\normalfont
  Our constraints \eqref{eq: constr for paramas of qsp 1}--\eqref{eq: constr for paramas of qsp 6} for the parameters ${\boldsymbol \varsigma}$ and ${\boldsymbol \kappa}$ are the same as those in \cite{BW21} except that we further require that $\overline{\varsigma_i} = \varsigma_i^{-1}$ as \cite{BW18}.
  This additional constraint is needed for Proposition \ref{prop: rho on Ui} below.
\end{rem}

For each $i \in I$, set
\[
  B_i := \begin{cases}
    F_i & \text{ if } i \in I_\bullet, \\
    F_i + \varsigma_i T_{w_\bullet}(E_{\tau(i)})K_i^{-1} + \kappa_i K_i^{-1} & \text{ if } i \in I_\circ.
  \end{cases}
\]

The \emph{$\imath$quantum group} associated with the datum above is the subalgebra $\mathbf{U}^\imath$ of the quantum group $\mathbf{U}$ generated by
\[
  \{ E_j, B_i, K_h \mid j \in I_\bullet,\ i \in I,\ h \in Y^\imath \}.
\]

\begin{rem}\normalfont
  In \cite{BW21}, it is assumed that either $I_\bullet = \emptyset$ or $|a_{i,j}| \leq 3$ for all $i,j \in I$.
  This assumption is used only to ensure the existence of bar-involution on $\mathbf{U}^\imath$.
  Now, we can remove this assumption due to the existence theorem \cite[Corollary 4.2]{Kol22} of bar-involution for arbitrary generalized Satake diagrams (see also \cite[Remark 3.6]{BW21}).
\end{rem}

\begin{prop}\label{prop: rho on Ui}
  The anti-algebra involution $\rho$ on $\mathbf{U}$ restricts to a one on $\mathbf{U}^\imath$$:$
  \[
    \rho(\mathbf{U}^\imath) = \mathbf{U}^\imath.
  \]
\end{prop}

\begin{proof}
  The argument in \cite[Section 4.2]{BW18} is still valid for our setting (see Remark \ref{rem: params for Ui}).
  Hence, the assertion follows from \cite[Proposition 4.16]{BW18}.
\end{proof}

\begin{cor}\label{cor: V(lm) as Ui-mod is comp red}
  Let $\lambda \in X^+$.
  Then, the integrable highest weight module $V(\lambda)$, regarded as a $\mathbf{U}^\imath$-module, is completely reducible.
\end{cor}

\begin{proof}
  Let $M \subseteq V(\lambda)$ be a $\mathbf{U}^\imath$-submodule, and $M^\perp$ the orthogonal complement of $M$ with respect to the nondegenerate contragredient bilinear form $(,)$ on $V(\lambda)$.
  By Proposition \ref{prop: rho on Ui}, we see that $M^\perp$ is a $\mathbf{U}^\imath$-submodule of $V(\lambda)$.
  Hence, the assertion follows.
\end{proof}

Let $\dot{\mathbf{U}}^\imath$ denote the modified $\imath$quantum group, ${}_\mathcal{A} \dot{\mathbf{U}}^\imath$ its integral form, and $\dot{\mathbf{B}}^\imath$ the $\imath$canonical basis of $\dot{\mathbf{U}}^\imath$.
For each $\zeta \in X^\imath$, let $1_\zeta \in \dot{\mathbf{U}}^\imath$ denote the corresponding idempotent.

\subsection{Weight Modules}
\begin{defi}\normalfont
  A $\mathbf{U}^\imath$-module $M$ is said to be a \emph{weight module} if it possesses a weight space decomposition, i.e., a vector space decomposition
  \[
    M = \bigoplus_{\zeta \in X^\imath} M_\zeta
  \]
  satisfying the following for all $\zeta \in X^\imath$.
  \begin{enumerate}
    \item $K_h m = q^{\langle h, \zeta \rangle} m$ for all $m \in M_\zeta$ and $h \in Y^\imath$.
    \item $E_j M_\zeta \subseteq M_{\zeta + \overline{\alpha_i}}$ for all $j \in I_\bullet$.
    \item $B_i M_\zeta \subseteq M_{\zeta - \overline{\alpha_i}}$ for all $i \in I$.
  \end{enumerate}
\end{defi}

A weight $\mathbf{U}$-module $M = \bigoplus_{\lambda \in X} M_\lambda$, regarded as a $\mathbf{U}^\imath$-module by restriction, is a weight $\mathbf{U}^\imath$-module with
\[
  M_\zeta = \bigoplus_{\substack{\lambda \in X \\ \overline{\lambda} = \zeta}} M_\lambda \quad \text{ for all } \zeta \in X^\imath.
\]

A weight $\mathbf{U}^\imath$-module $M = \bigoplus_{\zeta \in X^\imath} M_\zeta$ is equipped with the following $\dot{\mathbf{U}}^\imath$-module structure:
\[
  (x 1_\zeta) m := \delta_{\zeta, \eta} xm \quad \text{ for all } x \in \mathbf{U}^\imath,\ \zeta,\eta \in X^\imath,\ m \in M_\eta.
\]

Combining the two constructions above, we see that each weight $\mathbf{U}$-module has a natural $\dot{\mathbf{U}}^\imath$-module structure.
In particular, the $\mathbf{U}$-module $V(\lambda)$ with $\lambda \in X^+$ can be seen as a $\dot{\mathbf{U}}^\imath$-module such that
\[
  (x1_\zeta) v_\lambda = \delta_{\zeta, \overline{\lambda}} x v_\lambda \quad \text{ for all } x \in \mathbf{U}^\imath \text{ and } \zeta \in X^\imath.
\]

\begin{prop}[{{\it cf.}\ \cite[Lemma 6.2.1]{W21}}]\label{prop: hwm is a cyclic Ui mod}
  Let $M$ be a highest weight $\mathbf{U}$-module with highest weight vector $m$.
  Then, we have
  \[
    \dot{\mathbf{U}}^\imath m = M.
  \]
\end{prop}

\begin{proof}
  The assertion follows from \cite[Proposition 6.1]{Kol14}.
\end{proof}

\begin{prop}[{{\it cf.}\ \cite[Lemma 6.2.2]{W21}}]\label{prop: presen of Verma as Ui mod}
  Let $\lambda \in X^+$.
  Then, as a $\dot{\mathbf{U}}^\imath$-module, the Verma module $M(\lambda)$ is isomorphic to
  \[
    \dot{\mathbf{U}}^\imath 1_{\overline{\lambda}}/\sum_{j \in I_\bullet} \dot{\mathbf{U}}^\imath E_i 1_{\overline{\lambda}}.
  \]
\end{prop}

\begin{proof}
  The assertion follows from \cite[Proposition 6.2]{Kol14} and the argument in the proof of \cite[Lemma 6.2.2]{W21}.
\end{proof}

From now on, until the end of this subsection, assume that
\begin{align}\label{eq: assump on kappa}
  \kappa_i = 0 \ \text{ for all } i \in I_\circ.
\end{align}

\begin{prop}[{\textit{cf}.\ \cite[Lemma 6.2.3]{W21}}]\label{prop: presen for int hwm as Ui mod}
  Let $\lambda \in X^+$.
  For each $i \in I$, set $\lambda_i := \langle h_i, \lambda \rangle$.
  Then, as a $\dot{\mathbf{U}}^\imath$-module, the integrable highest weight module $V(\lambda)$ is isomorphic to
  \[
    \dot{\mathbf{U}}^\imath 1_{\overline{\lambda}}/\left(\sum_{i \in I} \dot{\mathbf{U}}^\imath B_{i, \overline{\lambda}}^{(\lambda_i+1)} + \sum_{j \in I_\bullet} \dot{\mathbf{U}}^\imath E_i 1_{\overline{\lambda}} \right),
  \]
  where
  \[
    B_{i,\overline{\lambda}}^{(\lambda_i+1)} := \begin{cases}
      \frac{1}{[\lambda_i+1]_i!}B_i^{\lambda_i+1} 1_{\overline{\lambda}} & \text{ if } i \notin I_\circ^{\tau, \bullet}, \\
      \frac{1}{[\lambda_i+1]_i!}\left( \prod_{k=0}^{\lambda_i} (B_i - [-\lambda_i + 2k]_i) \right) 1_{\overline{\lambda}} & \text{ if } i \in I_\circ^{\tau, \bullet}.
    \end{cases}
  \]
\end{prop}

\begin{proof}
  The proof is the same as \cite[Lemma 6.2.3]{W21} once we establish Propositions \ref{prop: hwm is a cyclic Ui mod} and \ref{prop: presen of Verma as Ui mod}.
\end{proof}

\begin{rem}\normalfont
  The elements $B_{i,\overline{\lambda}}^{(\lambda_i+1)}$ in Proposition \ref{prop: presen for int hwm as Ui mod} for $i \in I_\circ^{\tau, \bullet}$ are special cases of the $\imath$divided powers defined abstractly in \cite[Proposition 4.11]{BW18a}.
  Explicit formulas of the $\imath$divided powers were conjectured in \cite[Conjecture 4.13]{BW18a} and proved in \cite{BeW18}.
\end{rem}

\begin{lem}\label{lem: compare def rels for int hwm}
  Let $\lambda, \mu \in X^+$ be such that $\overline{\lambda} = \overline{\mu}$.
  Let $i \in I$ be such that $\langle h_i, \mu \rangle \leq \langle h_i, \lambda \rangle$.
  Then, we have
  \[
    \dot{\mathbf{U}}^\imath B_{i,\overline{\lambda}}^{(\langle h_i, \lambda \rangle+1)} \subseteq \dot{\mathbf{U}}^\imath B_{i,\overline{\mu}}^{(\langle h_i, \mu \rangle + 1)}.
  \]
\end{lem}

\begin{proof}
  It suffices to show that
  \[
    B_{i, \overline{\lambda}}^{\langle h_i, \lambda \rangle+1} \in \dot{\mathbf{U}}^\imath B_{i, \overline{\mu}}^{\langle h_i, \mu \rangle+1} \ \text{ for all } i \in I.
  \]
  When $i \notin I_\circ^{\tau, \bullet}$, the claim is clear.
  Hence, suppose that $i \in I_\circ^{\tau, \bullet}$.
  Since $\overline{\lambda} = \overline{\mu}$, there exists $\nu \in X$ such that
  \[
    \lambda - \mu = \nu + w_\bullet \tau(\nu).
  \]
  Hence, we have
  \[
    \langle h_i, \lambda-\mu \rangle = 2\langle h_i, \nu \rangle \in 2\mathbb{Z}.
  \]
  Then, the claim follows from the definition of $\imath$divided powers.
\end{proof}

\begin{prop}\label{prop: pii lm mu}
  Let $\lambda, \mu \in X^+$ be such that $\overline{\lambda} = \overline{\mu}$ and $\langle h_i, \mu \rangle \leq \langle h_i, \lambda \rangle$ for all $i \in I$.
  Then, there exists a unique $\mathbf{U}^\imath$-module homomorphism
  \[
    \pi^{\imath, \lambda}_\mu: V(\lambda) \rightarrow V(\mu)
  \]
  such that
  \[
    \pi^{\imath, \lambda}_\mu(v_\lambda) = v_\mu.
  \]
\end{prop}

\begin{proof}
  The assertion follows from Proposition \ref{prop: presen for int hwm as Ui mod} and Lemma \ref{lem: compare def rels for int hwm}.
\end{proof}

\begin{cor}\label{cor: suff cond for Vlm containing triv submod}
  Let $\lambda \in X^+$ be such that $\overline{\lambda} = \overline{0}$.
  Then, there exists a $\mathbf{U}^\imath$-submodule of $V(\lambda)$ isomorphic to the trivial module $V(0)$.
\end{cor}

\begin{proof}
  By Proposition \ref{prop: pii lm mu}, there exists a nonzero $\mathbf{U}^\imath$-module homomorphism from $V(\lambda)$ to $V(0)$.
  Since $V(\lambda)$ is completely reducible by Corollary \ref{cor: V(lm) as Ui-mod is comp red}, this implies the assertion.
\end{proof}

\begin{rem}\normalfont
  Assumption \eqref{eq: assump on kappa} can be weakened; for example, when $\kappa_i = [s_i]_i$ for some $s_i \in \mathbb{Z}$ for all $i \in I_\circ^{\tau, \bullet}$, we only need to modify $B_{i, \overline{\lambda}}^{(\lambda_i+1)}$ as
  \[
    \frac{1}{[\lambda_i+1]_i!} \left( \prod_{k=0}^{\lambda_i} (B_i - [s_i-\lambda_i+2k]_i) \right) 1_{\overline{\lambda}}.
  \] 
\end{rem}

\subsection{Stability conjecture}\label{subsect: stab conj}
The notions of based modules and based homomorphisms for the $\imath$quantum groups are defined in the same way as those for the quantum groups.

For each $\lambda, \mu, \nu \in X^+$, let $V^\imath(\lambda,\mu; \nu)$ denote the $\mathbf{U}^\imath$-submodule of $V(\lambda + \tau\nu) \otimes V(\mu + \nu)$ generated by the vector
\[
  v^\imath_{\lambda, \mu; \nu} := v_{w_\bullet(\lambda + \tau\nu)} \otimes v_{\mu + \nu}.
\]
We abbreviate $V^\imath(\lambda, \mu; 0)$ and $v^\imath_{\lambda, \mu; 0}$ as $V^\imath(\lambda, \mu)$ and $v^\imath_{\lambda, \mu}$ respectively.

By \cite[Proposition 7.1 (2)]{BW21}, there exists a unique $\mathbf{U}^\imath$-module homomorphism
\[
  \pi^\imath_{\lambda,\mu;\nu}: V^\imath(\lambda, \mu; \nu) \rightarrow V^\imath(\lambda, \mu)
\]
such that
\[
  \pi^\imath_{\lambda, \mu; \nu}(v^\imath_{\lambda, \mu; \nu}) = v^\imath_{\lambda, \mu}.
\]

For each $\lambda,\mu \in X^+$, by the definition of $V^\imath(\lambda,\mu;\nu)$, there exists a unique $\mathbf{U}^\imath$-module homomorphism
\[
  \dot{\pi}^\imath_{\lambda,\mu;\nu}: \dot{\mathbf{U}}^\imath \rightarrow V^\imath(\lambda,\mu;\nu);\ x \mapsto x v^\imath_{\lambda,\mu;\nu}.
\]

Now, we are ready to state the \emph{stability conjecture}.

\begin{conj}[{Stability Conjecture \cite[Remark 6.18]{BW18}}]\label{conj: stability}
  There exist parameters ${\boldsymbol \varsigma}$ and ${\boldsymbol \kappa}$ such that the $\mathbf{U}^\imath$-module homomorphisms $\dot{\pi}^\imath_{\lambda, \mu;\nu}$ are based for all $\lambda,\mu,\nu \in X^+$.
\end{conj}

\begin{rem}\normalfont
  By \cite[Theorem 7.2 (1)]{BW21} (see also \cite[Theorem 6.17 (1)]{BW18}), the stability conjecture is equivalent to say that there exist parameters ${\boldsymbol \varsigma}$ and ${\boldsymbol \kappa}$ such that the $\mathbf{U}^\imath$-module homomorphisms $\pi^\imath_{\lambda,\mu;\nu}$ are based for all $\lambda,\mu,\nu \in X^+$.
\end{rem}

Conjecture \ref{conj: stability} has been proved to be true for all admissible pairs of locally finite type with $I_\bullet = \emptyset$ in \cite[Theorem 6.2.8]{W21} by means of the $\imath$crystal base theory introduced in {\it loc.\ cit.}, and for all admissible pairs which are irreducible, of finite type, and of real rank $1$ in \cite[Corollary 4.6.3]{W23} by means of the (ordinary) crystal base theory and direct calculation based on the classification of such admissible pairs.

\section{Stability of $\imath$canonical bases}\label{sect: proof of stab conj}
In this section, we prove the stability conjecture (Conjecture \ref{conj: stability}) for all locally finite types.

First, we state a sufficient condition (Theorem \ref{thm: suff cond}) for the stability conjecture to be true.
Then, we state Conjecture \ref{conj: suff cond (strong)}, whose validity implies the sufficient condition above, and hence, the validity of the stability conjecture.
Finally, we prove Conjecture \ref{conj: suff cond (strong)} when the admissible pair is irreducible, of finite type, and of real rank $1$ in Subsection \ref{subsect: proof of conj for irr fin rk 1} and when the admissible pair is of locally finite type in Subsection \ref{subsect: prf loc fin}.

\subsection{Preliminaries}
In this subsection, the admissible pair is still arbitrary.

\begin{thm}\label{thm: suff cond}
  Let $\nu \in X^+$.
  Suppose that there exists a vector $w_0 \in \mathcal{L}(\nu + w_\bullet \tau(\nu))$ such that $\mathbf{U}^\imath w_0 \simeq V(0)$ and $w_0 \equiv_\infty v_{\nu + w_\bullet \tau(\nu)}$.
  Then, the $\mathbf{U}^\imath$-module homomorphism $\pi^\imath_{\lambda, \mu; \nu}$ is based for all $\lambda, \mu \in X^+$.
\end{thm}

\begin{proof}
  The assertion follows from \cite[Corollary 3.4.8]{W23} and the proof of \cite[Proposition 4.6.1]{W23}.
\end{proof}

\begin{lem}\label{lem: triv mod in cyclic mod}
  Let $M$ be a completely reducible weight $\mathbf{U}^\imath$-module generated by a weight vector $m \in M$ of weight $\zeta \in X^\imath$.
  Suppose that the multiplicity of the trivial module $V(0)$ in $M$ is nonzero.
  Then, we have $\zeta = \overline{0}$ and the multiplicity is $1$.
\end{lem}

\begin{proof}
  Let $N \subseteq M$ denote the sum of all $\mathbf{U}^\imath$-submodules of $M$ isomorphic to the trivial module.
  Let us write
  \[
    N = \bigoplus_{\omega \in \Omega} N^\omega, \quad N^\omega \simeq V(0)
  \]
  for some index set $\Omega$.

  Since $M$ is completely reducible, we have
  \[
    M = N \oplus L
  \]
  for some $\mathbf{U}^\imath$-submodule $L$.
  According to this decomposition, let us write
  \[
    m = m_N + m_L.
  \]

  Let us first show that $\zeta = \overline{0}$.
  Since $m \in M_\zeta$, we have
  \[
    m = 1_\zeta m = 1_\zeta m_N + 1_\zeta m_L.
  \]
  Noting that $1_\zeta m_N \in N$ and $1_\zeta m_L \in L$, we see that
  \[
    1_\zeta m_N = m_N.
  \]
  On the other hand, since $N$ isomorphic to a sum of copies of the trivial module, we have
  \[
    1_\zeta m_N = \delta_{\zeta, \overline{0}} m_N.
  \]
  Therefore, we obtain
  \[
    \zeta = \overline{0}.
  \]

  Next, let us prove that $|\Omega| = 1$.
  Let us write
  \[
    m_N = \sum_{\omega \in \Omega} n^\omega, \quad n^\omega \in N^\omega.
  \]
  Let us fix $\omega_0 \in \Omega$.
  We shall show that $\omega_0 = \omega$ for all $\omega \in \Omega$.

  Since $M$ is cyclic, we must have $n^\omega \neq 0$, and there exists $x^\omega \in \mathbf{U}^\imath$ such that $xm = n^\omega$ for all $\omega \in \Omega$.
  As before, we see that $x^{\omega_0} n^\omega = \delta_{\omega_0, \omega} n^\omega$.
  On the other hand, there exists a $\mathbf{U}^\imath$-module isomorphism $f_{\omega_0,\omega}: N^{\omega_0} \rightarrow N^\omega$ mapping $n^{\omega_0}$ to $n^\omega$.
  Then, we have
  \[
    x^{\omega_0} n^{\omega} = f_{\omega_0, \omega}(x^{\omega_0} n^{\omega_0}) = f(n^{\omega_0}) = n^\omega.
  \]
  This implies that $\omega_0 = \omega$, as desired.
  Thus, the proof completes.
\end{proof}

\begin{prop}\label{prop: triv mod generator}
  Let $\lambda \in X^+$.
  The $\mathbf{U}^\imath$-module $V(\lambda)$ has a submodule isomorphic to the trivial module $V(0)$ if and only if $\overline{\lambda} = 0$.
  Moreover, if this is the case, then there exists a unique $($up to nonzero scalar multiple$)$ vector $w_0 \in V(\lambda)$ such that
  \[
    \mathbf{U}^\imath w_0 \simeq V(0).
  \]
\end{prop}

\begin{proof}
  The assertion follows from Corollary \ref{cor: suff cond for Vlm containing triv submod} and Lemma \ref{lem: triv mod in cyclic mod}; we can apply them since $V(\lambda)$ is completely reducible as a $\mathbf{U}^\imath$-module by Corollary \ref{cor: V(lm) as Ui-mod is comp red}.
\end{proof}

\begin{conj}\label{conj: suff cond (strong)}
  There exist parameters ${\boldsymbol \varsigma}$ and ${\boldsymbol \kappa}$ such that the following hold.
  Let $\lambda \in X^+$ be such that $\overline{\lambda} = \overline{0}$, and $w_0 \in V(\lambda)$ the vector in Proposition $\ref{prop: triv mod generator}$.
  Then, there exists $c \in \mathbb{Q}(q)^\times$ such that
  \[
    cw_0 \equiv_\infty v_\lambda.
  \]
\end{conj}

\begin{rem}\normalfont
  If Conjecture \ref{conj: suff cond (strong)} is true, then Theorem \ref{thm: suff cond} implies the validity of Conjecture \ref{conj: stability}.
\end{rem}

\subsection{Irreducible finite types of real rank one}\label{subsect: proof of conj for irr fin rk 1}
In this subsection, we assume that the admissible pair $(I_\bullet, \tau)$ is irreducible, of finite type, and of real rank $1$ (see Definition \ref{def: related to adm pair}).
There are only eight types of such admissible pairs: $A\mathrm{I}_1$, $A\mathrm{II}_3$, $A\mathrm{III}_2$, $A\mathrm{IV}_n$, $B\mathrm{II}_n$, $C\mathrm{II}_{n}$, $D\mathrm{II}_{n}$, $F\mathrm{II}_4$.

Set
\[
  Y' := \sum_{i \in I} \mathbb{Z} h_i \subseteq Y, \quad X' := \operatorname{Hom}_\mathbb{Z}(Y' ,\mathbb{Z}).
\]
Then, there exists a group homomorphism
\[
  X \rightarrow X';\ \lambda \mapsto \lambda'
\]
defined by
\[
  \langle h', \lambda' \rangle = \langle h', \lambda \rangle \ \text{ for all } h' \in Y'.
\]
Set
\[
  \Pi' := \{ \alpha'_i \mid i \in I \}.
\]
Then, $(Y', X', \langle ,  \rangle, \Pi^\vee, \Pi')$ is a Satake datum of type $(I_\bullet, \tau)$.
Let $(\mathbf{U}', \mathbf{U}^{\imath \prime})$ denote the corresponding quantum symmetric pair.

For each $i \in I$, let $\varpi_i \in X'$ denote the fundamental weight:
\[
  \langle h_j, \varpi_i \rangle = \delta_{i,j} \ \text{ for all } j \in I.
\]

\begin{lem}\label{lem: proof of conj for irr fin rk 1}
  Conjecture $\ref{conj: suff cond (strong)}$ is true for $(\mathbf{U}', \mathbf{U}^{\imath \prime})$.
\end{lem}

\begin{proof}
  Let us fix a total order $\preceq$ on $I_\circ$.
  For each $i \in I_\circ$, set $\kappa_i = 0$ and 
  \[
    \varsigma_i = \begin{cases}
      q^{-1} & \text{ if $(I_\bullet, \tau)$ is of type $A\mathrm{I}_1$}, \\
      q & \text{ if $(I_\bullet, \tau)$ is of type $A\mathrm{II}_3$}, \\
      1 & \text{ if $(I_\bullet, \tau)$ is of type $A\mathrm{III}_2$}, \\
      1 & \text{ if $(I_\bullet, \tau)$ is of type $A\mathrm{IV}_n$ and $i \prec \tau(i)$}, \\
      (-1)^n q^{n-1} & \text{ if $(I_\bullet, \tau)$ is of type $A\mathrm{IV}_n$ and $i \succ \tau(i)$}, \\
      q^{2n-3} & \text{ if $(I_\bullet, \tau)$ is of type $B\mathrm{II}_n$}, \\
      q^{n-1} & \text{ if $(I_\bullet, \tau)$ is of type $C\mathrm{II}_n$}, \\
      q^{n-2} & \text{ if $(I_\bullet, \tau)$ is of type $D\mathrm{II}_n$}, \\
      q^5 & \text{ if $(I_\bullet, \tau)$ is of type $F\mathrm{II}_4$}.
    \end{cases}
  \]
  Note that the parameters ${\boldsymbol \varsigma}$ and ${\boldsymbol \kappa}$ are the same as in \cite[Section 4]{W23}.

  Let $\varpi \in X^+$ be as in \cite[Section 4.6]{W23}:
  \[
    \varpi := \begin{cases}
      2\varpi_i & \text{ if $(I_\bullet, \tau)$ is of type $A\mathrm{I}_1$}, \\
      \varpi_i+\varpi_{\tau(i)} & \text{ if $(I_\bullet, \tau)$ is either of type $A\mathrm{III}_2$ or $A\mathrm{IV}_n$}, \\
      \varpi_i & \text{ if $(I_\bullet, \tau)$ is either of type $A\mathrm{II}_3$, $B\mathrm{II}_n$, $C\mathrm{II}_n$, $D\mathrm{II}_n$, or $F\mathrm{II}_n$},
    \end{cases}
  \]
  where $i$ is an element of $I_\circ$.
  Then, we see that
  \[
    \{ \lambda \in X^+ \mid \overline{\lambda} = \overline{0} \} = \mathbb{Z}_{\geq 0} \varpi.
  \]
  Now, the assertion follows from \cite[Proposition 4.6.1]{W23}.
\end{proof}

\begin{prop}\label{prop: proof of conj for irr fin rk 1 general}
  Conjecture \ref{conj: suff cond (strong)} is true for $(\mathbf{U}, \mathbf{U}^\imath)$.
\end{prop}

\begin{proof}
  Set the parameters ${\boldsymbol \varsigma}$ and ${\boldsymbol \kappa}$ is in the proof of Lemma \ref{lem: proof of conj for irr fin rk 1}.

  Since $\overline{\lambda} = \overline{0}$, there exists $\nu \in X$ such that
  \[
    \lambda = \nu + w_\bullet \tau(\nu).
  \]
  Then, we obtain
  \[
    \lambda' = \nu' + w_\bullet \tau(\nu').
  \]

  Let us regard $V(\lambda)$ as a $\mathbf{U}'$-module via the canonical algebra homomorphism $\mathbf{U}' \rightarrow \mathbf{U}$.
  Then, we see that it is isomorphic to $V(\lambda')$ and that $\mathbf{U}^{\imath \prime} w_0$ is a $\mathbf{U}^{\imath \prime}$-submodule of $V(\lambda')$ isomorphic to the trivial module.
  Now, the assertion follows from Lemma \ref{lem: proof of conj for irr fin rk 1}.
\end{proof}

\begin{cor}\label{cor: characterization of triv mod at infty for irr fin rk 1}
  Let $M$ be an integrable $\mathbf{U}$-module with a crystal base $(\mathcal{L}, \mathcal{B})$.
  Let $m \in \mathcal{L}$ be such that $\mathbf{U}^\imath m \simeq V(0)$.
  Then, we have $m \in \bigoplus_{\substack{\lambda \in X \\ \overline{\lambda} = \overline{0}}} M_{\lambda}$, and
  \[
    \tilde{E}_i m \equiv_\infty 0 \ \text{ for all } i \in I.
  \]
\end{cor}

\begin{proof}
  By Theorem \ref{thm: comp red for fin}, the integrable $\mathbf{U}$-module $M$ is completely reducible with irreducible components of the form $V(\lambda)$ with various $\lambda \in X^+$.
  Hence, the assertion is an immediate consequence of Proposition \ref{prop: proof of conj for irr fin rk 1 general} (see Example \ref{ex: characterization of the hwv at infty}).
\end{proof}

\subsection{Locally finite types}\label{subsect: prf loc fin}
For each $i \in I_\circ$ recall the subdatum $I_i \subseteq I$ from Definition \ref{def: related to adm pair}, and set
\[
  \Pi_i := \{ h_j \mid j \in I_i \}, \quad \Pi^\vee_i := \{ \alpha_j \mid j \in I_i \}.
\]
Then, $(Y, X, \langle ,  \rangle, \Pi_i, \Pi^\vee_i)$ is a Satake datum of type $(I_{i, \bullet}, \tau|_{I_i})$.
Let $(\mathbf{U}_i, \mathbf{U}^\imath_i)$ denote the corresponding quantum symmetric pair.
We regard $\mathbf{U}_i$ as a subalgebra of $\mathbf{U}$ in the canonical way.

\begin{thm}\label{thm: proof of conj for loc fin}
  Suppose that the admissible pair is of locally finite type.
  Then, Conjecture $\ref{conj: suff cond (strong)}$ is true.
\end{thm}

\begin{proof}
  Set the parameters ${\boldsymbol \varsigma}$ and ${\boldsymbol \kappa}$ as in the proof of Lemma \ref{lem: proof of conj for irr fin rk 1}.

  Without loss of generality, we may assume that $w_0 \in \mathcal{L}(\lambda) \setminus q^{-1} \mathcal{L}(\lambda)$.

  First, suppose that $i \in I_\bullet$.
  Then, we have
  \[
    E_i w_0 = F_i w_0 = 0.
  \]
  This implies that
  \[
    \tilde{E}_i w_0 = 0.
  \]

  Next, suppose that $i \in I_\circ$.
  Since the $\mathbf{U}$-module $V(\lambda)$, as a $\mathbf{U}_i$-module, is integrable, we may apply Corollary \ref{cor: characterization of triv mod at infty for irr fin rk 1} to obtain
  \[
    \tilde{E}_i w_0 \equiv_\infty 0.
  \]

  Now, the assertion follows from Example \ref{ex: characterization of the hwv at infty}.
\end{proof}

\end{document}